\documentclass[12pt,twoside]{amsart}
\usepackage{latexsym,amsmath,amsopn,amssymb,amsthm,amsfonts}
\usepackage[T2A]{fontenc}
\usepackage[cp1251]{inputenc}
\usepackage[russian,english]{babel}
\usepackage{graphicx}

\newtheorem{theorem}{Theorem}

\newtheorem{problem}{Problem}

\newtheorem{remark}{Remark}

\newtheorem{definition}{Definition}

\newtheorem{corollary}{Corollary}

\newtheorem{lemma}{Lemma}

\newcommand{\CR}{\operatorname{CR}}

\begin{document}

\title[homogeneous spaces with inner metric]{Locally compact homogeneous 
spaces with inner metric}
\author{V.\,N.\,Berestovskii}
\thanks{The author was partially supported by the Russian Foundation 
for Basic Research (Grant 14-01-00068-a) and a grant of the Government 
of the Russian Federation for the State Support of Scientific Research 
(Agreement \No 14.B25.31.0029)}
\address{Sobolev Institute of Mathematics SD RAS, Novosibirsk}
\address{V.N.Berestovskii}
\email{vberestov@inbox.ru}
\maketitle
\maketitle {\small
\begin{quote}
\noindent{\sc Abstract.}
The author reviews his results on locally compact homogeneous spaces 
with inner metric, in particular, homogeneous manifolds with inner metric. 
The latter are isometric to homogeneous (sub-)Finslerian manifolds; under 
some additional conditions they are isometric to homogeneous 
(sub)-Riemannian manifolds. The class $\Omega$ of all locally compact 
homogeneous spaces with inner metric is supplied with some metric $d_{BGH}$ 
such that 1) $(\Omega,d_{BGH})$ is a complete metric space; 2) a sequences 
in $(\Omega,d_{BGH})$ is converging if and only if it is converging in 
Gromov-Hausdorff sense; 3) the subclasses $\mathfrak{M}$ of homogeneous 
manifolds with inner metric and $\mathfrak{LG}$ of connected Lie groups with 
left-invariant Finslerian metric are everywhere dense in $(\Omega,d_{BGH}).$ 
It is given a metric characterization of Carnot groups with left-invariant 
sub-Finslerian metric. At the end are described homogeneous manifolds such 
that any invariant inner metric on any of them is Finslerian. 
\end{quote}}

{\small
\begin{quote}
\end{quote}}

\section*{Introduction}

One can observe in last decades an intensive development of non-holonomic metric
geometry and its applications to geometric group theory, analysis, 
$\CR$-manifolds, the theory of hypo-elliptic differential equations, 
non-holonomic mechanics, mathematical physics, thermodynamics, 
neurophysiology of vision etc.  R.Montgomery's book \cite{Mont} gives 
a well written track of this. A natural context for (sub-)Finslerian, 
in particular, (sub-)Riemannian geometry is geometric control theory 
\cite{Jur}, \cite{AS}.

Homogeneous Riemannian and Finslerian manifolds and their 
non-holonomic generalizations, homogeneous sub-Riemannian and 
sub-Finsleian manifolds, are especially important as models, 
because in some cases it is possible to find exactly geodesics, 
shortest arcs, conjugate and cut locus, and even distances for them.   

A simple geometric axiomatic for homogeneous (sub-)Finslerian, 
in particular, (sub-)Riemannian, manifolds in general context of 
locally compact homogeneous spaces with inner metric, have been 
announced in paper \cite{Ber-89.2}. Later appeared proofs of 
this announcement \cite{Ber-88}, \cite{Ber-89}, \cite{Ber-89.1}, 
other interpretations, my later results and their survey in
\cite{BGorb-14}, corrections to proofs of some results cited in \cite{Ber-89.2}.
Also last years some colleagues exhibited an interest in my old results from  
\cite{Ber-89.2}. I am very obliged to professor D.V.Alekseevsky for 
useful discussions on this matter. Hopefully, all this serves as 
enough motivation to present a modified, renewed, relatively short 
(with omission of well-known definitions), version of some statements 
from \cite{Ber-89.2} and \cite{BGorb-14}.

\section{Locally compact homogeneous spaces with inner metric}

Let us remind main definitions. \textit{A path} in a topological 
space $X$ is a continuous map of some closed bounded
interval of the real line to the space $X$. A metric space is 
called \textit{the space with inner metric}, if the distance
between any two its points is equal to the infimum of the length 
of paths joining these points. A metric space
is \textit{homogeneous} if its isometry group acts transitively, 
i.e., for any two points in the space there is an isometry 
(motion) of the space moving one of these points to the other.

The \textit{S.~E.~Cohn-Vossen theorem} \cite{CV} states that every 
locally compact complete space $(M,\rho)$ with inner
metric is \textit{finitely compact}, i.e., any closed bounded 
subset in $(M,\rho)$ (in particular, any closed ball
$B(x,r)$ of radius $r$ with the center at $x$) is compact; moreover, 
the space $(M,\rho)$ is \textit{geodesic}. The last
statement means that any two points of the space can be joined 
by a \textit{segment} or \textit{shortest arc}, i.e., a curve (path) 
of length which is equal to the distance between these points.

Further we suppose that $(M,\rho)$ is  an arbitrary locally compact 
homogeneous space with inner metric $\rho$, and $G=I(M)$ is its 
motion group with compact-open topology with respect to its action 
on $(M,\rho)$, $G_0$ is a connected component of the unit in the 
group $G.$ In view of homogeneity and the local compactness, the 
space $(M,\rho)$ is metrically complete, the S.~E.~Cohn-Vossen 
theorem holds, and so we can use a shorter term 
''\textit{locally compact homogeneous geodesic space}''.

The \textit{Busemann metric}, see \cite{Bu1},
$$\delta_p(f,g)=\sup_{x\in M}\rho(f(x),g(x))e^{-\rho(p,x)},\quad\mbox{where}\quad p\in M,$$ is introduced on the group $G$.

The following results are proved in \cite{Ber}. The metric $\delta_p$ 
depends on the choice of the point $p\in M$, but it is bi-Lipschitz 
equivalent to the metric $\delta_q$ for any point $q\in M,$ and thus, 
independently on the point $p\in M,$ defines a topology $\tau$, which 
coincides with the compact-open topology on $G$ with respect to
its action on $(M,\rho)$. Let us remind that the subbasis of the 
compact-open topology consists of sets
$G(K;U):=\{g\in G| g(K)\subset U\}$, where $K$ is a compact and $U$ 
is an open subset in $M$. The metric $\delta_p$ is invariant under 
the left translations by elements of the group $G$ and under the 
right translations by elements of its compact subgroup $H$, the 
stabilizer of the groups $G$ at a point $p$; for natural 
identification of $G/H$ with $M,$ defined by formula 
$\sigma(gH)=g(p),$ the quotient metric $\Delta_p$ on $G/H,$ 
induced by the metric $\delta_p,$ is equivalent to the metric 
$\rho,$ the metric space $(G,\delta_p)$ is locally compact, complete, 
and separable. The pair $(G,\tau)$ is a topological group acting 
continuously and properly on the left on $(M,\rho)$ by isometries.
The subgroup $G_0$ is transitive on $M.$ 

The following interesting problem is still open.

\begin{problem}
\cite{BerP}
Is it true, that in the general case, the connected group $G_0$ 
or another  transitive  on $M$ closed connected subgroup of
the group $G$ is locally connected or, which is equivalent, 
locally arcwise connected?
\end{problem}

Observe in relation with this, that the paper \cite{BerP}  gives 
a very short proof of a new (at that time) result, that is
a global form of a theorem on the local representation of a 
group as a direct product coming from the Iwasawa-Gleason-Yamabe theory
\cite{Iw, Gl, Ya} for locally compact groups.

\begin{theorem}
\label{glob}
Let $G$ be a connected locally compact (Hausdorff) topological group. 
Then there exists a compact subgroup $K\subset G,$ a connected, 
simply connected Lie group $L,$ and  a surjective local isomorphism 
$\pi: K\times L\rightarrow G.$ Furthermore, if $G$ is locally connected, 
then $K$ is connected and locally connected, and $\pi$ is a covering epimorphism.
\end{theorem}

The following characterization of locally compact homogeneous 
geodesic spaces as homogeneous spaces of topological groups \cite{BerP} holds.

\begin{theorem}
\label{BP}
Every locally compact homogeneous geodesic space is isometric to 
some locally compact locally connected quotient space
$G/H$ of a connected locally compact topological group $G$ 
with the first countability axiom,  by a compact subgroup $H,$
endowed with a $G$-invariant geodesic metric.

Conversely, every locally connected, locally compact homogeneous 
quotient space $G/H$ of a connected locally compact topological
group $G$ with the first countability axiom  by a compact 
subgroup $H$ admits a  $G$-invariant geodesic metric $\rho$.
\end{theorem}

\begin{corollary}
A locally compact topological group $(G,\tau)$ admits some 
left-invariant geodesic metric if and only if $(G,\tau)$ is 
connected, locally connected, and satisfies the first countability axiom.
\end{corollary}

\begin{theorem}
\label{ya}
\cite{Ya}
Any neighbourhood $U$ of the unit $e$ in a connected locally compact
topological group  $G$ contains closed (even compact) normal subgroups 
$N=N_U$ with the quotient group $G/N,$ which is a (connected) Lie group.
\end{theorem}

\begin{lemma}
\label{gl}
\cite{Glu} 
If $N_1$ and $N_2$ are
normal subgroups of a locally compact topological group $G$ such 
that $G/N_1$ and $G/N_2$ are Lie groups, then
$G/(N_1\cap N_2)$ is also a Lie group.
\end{lemma} 

We need the following definition in order to formulate other 
structural results on locally compact homogeneous geodesic spaces.

\begin{definition}
\label{subm}
\cite{BG-00}
A map of metric spaces $f\colon M\rightarrow N$ is said to be 
\textit{submetry}, if for any point $x\in M$, and any number
$r>0,$ we have $f(B_M(x,r))=B_N(f(x),r).$ Here $B$ denotes 
the closed ball of corresponding radius in the corresponding space.
\end{definition}

\begin{theorem}
\cite{BG-00}
Any Riemannian submersion of complete smooth Riemannian manifolds 
is a submetry. Conversely, a submetry of smooth Riemannian manifolds 
is the Riemannian submersion of class $C^{1,1}$.
\end{theorem}

On the ground of theorems \ref{BP}, \ref{ya}, lemma \ref{gl}, and definition
\ref{subm}, we prove the following

\begin{theorem}
\label{invlim}
\cite{Ber-89.2, Ber-89}
A metric space $(M,\rho)$ is a locally compact homogeneous geodesic 
space, if and only if, it can be represented as the
inverse metric limit of some sequence $(M_n=(G/N_n)/(HN_n/N_n),\rho_n),$
where $N_n$ is non-increasing sequence of compact normal subgroups of $G$ 
such that $\cap_{n=1}^{\infty}N_n=\{e\}$, of 
homogeneous geodesic manifolds bound by the proper
(the preimage of a compact set is compact) submetries 
$p_{nm}\colon (M_m,\rho_m)\rightarrow (M_n,\rho_n), n\leq m,$ and
$p_{n}:(M,\rho)\rightarrow (M_n,\rho_n),$ where $p_{n}=p_{nm}\circ p_m$ and
$p_{ns}=p_{nm}\circ p_{ms}$ if $n\leq m\leq s.$

This means that (non-decreasing) functions $\rho_n\circ (p_n\times p_n)$ 
uniformly converge to the metric $\rho.$ Under this condition, 
$p_{nm}\in C^{\infty}$, and one can assume that the fibers of 
these submetries are connected. 
\end{theorem}

In some sense, Theorem~\ref{invlim} reduces the study of 
locally compact homogeneous geodesic spaces to the case
of homogeneous geodesic manifolds.

Let us recall, that  the \textit{Hausdorff distance} $d_H(A,B)$ between 
two bounded subsets of an arbitrary metric space $M$ is
the infimum of positive numbers $r$, such that $A$ is contained 
in the $r$-neighbourhood of the set $B$, and
$B$ is contained in the $r$-neighbourhood of the set $A.$ The 
pair $(K(M),d_H)$  is a metric space where $K(M)$ is the family of all closed
bounded subsets of the metric space $M$. It is complete if the 
space $M$ is complete \cite{Kur}.

\begin{definition}
\label{GH}
\textit{The Gromov-Hausdorff distance} $d_{GH}(A,B)$ between compact 
metric spaces is defined as the infimum of all distances $d_H(f(A),g(B))$ 
for all metric spaces $M$, and for all isometric embeddings 
$f\colon A\rightarrow M,$ and $g: B\rightarrow M.$ By definition, a 
sequence $((X_n,x_n),\rho_n)$ of finitely compact complete spaces with
metrics  $\rho_n$ and chosen points $x_n$ 
\textit{Gromov-Hausdorff-converges} to a similar space $((X,x),\rho),$
if for any number $r>0,$ the distance 
$$d_{GH}(B_{X_n}(x_n,r),B_X(x,r))\rightarrow 0,$$ as $n\rightarrow +\infty.$
\end{definition}

\begin{definition}
\label{dist}
The distance $d_{BGH}$ between finitely compact metric spaces with 
chosen points $(X,x)$ and $(Y,y)$ is equal by definition to
\begin{equation}
\label{BGH}
d_{BGH}((X,x),(Y,y))=\sup_{r\geq 0}d_{GH}(B_{X}(x,r),B_{Y}(y,r))e^{-r}.
\end{equation}
\end{definition}

As a consequence of S.~E.~Cohn-Vossen theorem, cited above, this 
definition is applicable to locally compact complete
spaces with inner metric, in particular, to locally compact 
homogeneous spaces with inner metric. It is clear that
in the latter case the distance $d_{BGH}$ does not depend on 
the choice of points $x\in X$ and $y\in Y$.  Let $\Sigma$
and $\Theta$ be respectively the classes of all finitely compact 
metric spaces and locally compact complete inner
metric spaces with chosen points, and let $\Omega$ be a class 
of all locally compact homogeneous spaces with inner metric.

\begin{theorem}
\label{space}
The pair $(\Sigma,d_{BGH})$ is a complete metric space. The 
convergence of sequences in this metric space is equivalent to 
the Gromov-Hausdorff convergence. So $\Theta$ and $\Omega$ are 
closed subspaces of $(\Sigma,d_{BGH}).$ Moreover, 
the subclass $\mathfrak{M}$ of homogeneous manifolds with 
inner metric is everywhere dense in $(\Omega,d_{BGH}).$ 
\end{theorem}

\section{Homogeneous manifolds with inner metric}

\begin{theorem}
\label{man}
The following statements for locally compact homogeneous space with inner metric 
$(M,\rho)$ are equivalent:

(1) $M$ is a (connected) topological manifold;

(2) $M$ has finite topological dimension;

(3) $M$ is locally contractible;

(4) $(M,\rho)$ is isometric to $(G/H,d),$ where $G$ is a connected Lie group,
$H$ is a compact Lie subgroup of $G$, and $d$ is some inner metric 
on $G/H,$ invariant relative to the canonical left action of $G$ on $G/H.$ 
\end{theorem}

Let us give some explanations. Evidently, (1) implies (2) and (3); (4) implies
other statements. Now (2), theorems \ref{BP} and \ref{invlim} imply 
that $M=G/H,$ where $G$ is a locally compact, connected topological 
group such that some neighborhood $U$ of $e$ contains no nontrivial 
normal subgroup, and H is a compact subgroup of $G$. Then theorem 
\ref{ya} implies that $G$ is a connected Lie group,
$H$ is a compact Lie subgroup of $G,$ which proves (4).
  
The statement (3) together with theorem \ref{BP} would imply the 
statement (4) by J.~Szenthe's claim in \cite{Sze}:
\textit{let a $\sigma$-compact locally compact group $G,$ with 
a compact quotient $G/G_0,$ acts continuously (and properly) as 
a transitive and faithful transformation group on a locally 
contractible space $X.$ Then $X$ is a manifold and $G$ is a Lie group.}

However, it was discovered in 2011 by S.~Antonyan \cite{Ant}, that 
Szenthe's proof of this claim contains a serious gap.
Independently  Szenthe's claim was proved by S.~Antonyan and 
T.~Dobrowolski \cite{AD}, by K.~H.~Hoffmann and L.~Kramer
\cite{HK}, see also the book by K.H.Hoffmann and S.A.Morris 
\cite{HM}, pp. 592--605. 

Theorem \ref{man} gives topological characterization of homogeneous manifolds
with inner metric. Now we shall describe their metric structure.

Let $M=G/H$ be the quotient manifold of a connected Lie group $G$ by its
compact Lie subgroup $H;$ $\mathfrak{g}=G_e,$ $\mathfrak{h}=H_e$ be Lie algebras
of Lie groups $G,$ $H.$ Let us set the following objects:

(a1) $L_e$ is $Ad(H)$-invariant vector subspace of $\mathfrak{g}$ such that
$\mathfrak{h}\subset L_e$ and $\mathfrak{g}$ is the least Lie subalgebra of
$\mathfrak{g}$ which contains $L_e;$

(a2) $D_H=dp(e)(L_0),$ where $p: G\rightarrow G/H$ is the canonical projection and
$dp$ is its differential;

(a3) $F_H$ is a norm on $D_H$ which is invariant relative to the 
(linear) isotropy group of $G/H$ at $H\in G/H;$

(a4) $D$ is $G$-invariant distribution on $G/H$ such that $D(H)=D_H;$

(a5) $F$ is $G$-invariant norm on the distribution $D$ such that $F(H)=F_H.$ 

\begin{theorem}
\label{mman1}
\cite{Ber-89.2}, \cite{Ber-88},\cite{Ber-89.1}
Let $M=G/H$ be the quotient space of a connected Lie group $G$ by its compact Lie 
subgroup $H,$ $(D,F)$ is a pair with conditions (a1)---(a5). Then the formula
\begin{equation}
\label{dist}
d_c(x,y)=\inf _c \int_0^1 F(\dot{c}(t))dt,
\end{equation}
where $c=c(t),$ $0\leq t\leq 1,$ are arbitrary piecewise smooth paths in
$G/H,$ tangent to distribution $D$ and joining points $x$ and $y$ from $G/H,$
defines some $G$-invariant geodesic metric $d_c$ on $G/H$ (compatible with the
standard topology on $G/H$). 
\end{theorem}

\begin{remark}
Conditions (a1), (a2), and (a4) imply that the distribution $D$ 
from theorem \ref{mman1} is \textit{completely nonholonomic} \cite{Rash}. Therefore
any two points $x$ and $y$ from $G/H$ can be joined by some piecewise smooth
path by Rashevsky-Chow theorem \cite{Rash}, \cite{Ch}, so $d_c(x,y)$ is finite.
If $D = TM$ (respectively, $D\neq TM$) then $d_c$ is said to be 
(sub-)Finslerian and (sub-)Riemannian if additionally $F_H$ is an Euclidean
norm on $D_H.$ Note that a norm $\|\cdot\|$ on a vector space $V$ is Euclidean
if and only if 
$$\|a+b\|^2+ \|a-b\|^2=2(\|a\|^2+\|b\|^2)\quad\mbox{for every}\quad a,b\in V.$$
\end{remark}

\begin{theorem}
\label{mman}
\cite{Ber-88},\cite{Ber-89.1}
Every $G$-invariant inner metric $\rho$ on a homogeneous manifold 
$G/H$ of connected Lie group $G$ by its compact subgroup $H$ is 
sub-Finslerian or Finslerian. In addition, the Lie group $G$ admits 
a left-invariant sub-Finslerian or Finslerian (sub-Riemannian or 
Riemannian if $\rho$ is sub-Riemannian or Riemannian) metric $\rho_0$ 
such that the canonical projection $p:(G,\rho_0)\rightarrow (G/H,\rho)$ 
is submetry. 
\end{theorem}

Using this theorem, it is not difficult to prove the following addition to theorem
\ref{space}.

\begin{theorem}
\label{add}
\cite{BGorb-14}
The class $\mathfrak{LG}$ of connected Lie groups with left-invariant Finslerian
metric is everywhere dense in $(\Omega,d_{BGH}).$ 
\end{theorem}

The last statement of theorem \ref{mman} implies that the search of geodesics and
shortest arcs of invariant (sub-)Finslerian or (sub-)Riemannian metric 
on homogeneous manifolds reduces in many respects to the case of Lie 
groups with left-invariant (sub-)Finslerian or (sub-)Riemannian metric.

Any shortest arc, parametrized by the arc-length, on $(G,\rho_0)$ from 
theorem \ref{mman} is a solution of a \textit{time-optimal problem}; so 
it necessarily satisfies the \textit{Pontryagin maximum principle} 
(PMP) \cite{Pont}, \cite{Ber-89.2}. Unfortunately, this principle is 
useful only for so-called \textit{normal} shortest arcs and geodesics, 
when a maximum, supplied for them by PMP, is positive. Every normal 
geodesic on $(G,\rho_0)$ is smooth if $\rho_0$ is sub-Riemannian metric; 
moreover, if any geodesic on $(G,\rho_0)$ is normal (which is always 
true if $\rho_0$ is Riemannian) then any geodesic on $(G/H,\rho)$ is smooth. 

Let us note that using PMP, the author found in paper \cite{Ber-94} 
all geodesics and shortest arcs of arbitrary left-invariant 
sub-Finslerian metric on three-dimensional Heisenberg group.  

\section{Tangent cones and Carnot groups}

\begin{definition}
\label{sim}
A bijection of metric space $(M,\rho)$ onto itself is called a 
\textit{(metric) $a$-similarity}, if $\rho(f(x),f(y))=a\rho(x,y)$ 
for all points $x,y \in M$, where
$a\in \mathbb{R}, a>0$. The $a$-similarity is called
\textit{nontrivial}, if $a\neq 1.$
\end{definition}

\begin{theorem}
\label{ccg} \cite{BV-92}, \cite{Ber-2004}
A locally compact homogeneous space with an inner metric $(M,\rho)$ 
admits nontri\-vi\-al metric similarities if and only
if  $(M,\rho)$ is isometric to a finite-dimensional normed vector 
space or to a Carnot group, i.e. connected, simply connected,
(noncommutative) nilpotent  stratified  Lie group $C$ with the Lie algebra 
$LC=L=\bigoplus_{k=1}^{m}L_k$ (of nilpotentness
depth $m>1$), which is a direct sum of vector subspaces $L_k\subset L$ 
under the conditions $L_{i+1}=[L_i,L_1];\quad L_k=0 \quad\mbox{if}\quad k>m;$ 
with left-invariant sub-Finslerian metric $d_{cc},$ defined by a 
left-invariant norm $F_c$ on the left-invariant distribution $\Delta,$ 
where $\Delta(e)=L_1.$ Moreover, $(M,\rho)$ admits $a$-similarities 
for all positive $a\in \mathbb{R}.$
\end{theorem}

\begin{theorem}
\label{tc}
\cite{Ber-90}
If $(M,\rho)$ is a homogeneous manifold with inner metric then at any point 
$x\in (M,\rho)$ there exists the tangent cone $\tau_xM$ to the 
manifold $(M,\rho)$ (in the Gromov's sense \cite{Gro}) as the 
Gromov-Hausdorff limit of spaces
$((M,x),\alpha \rho)$ when $\alpha \rightarrow +\infty.$ Let suppose that 
$(M,\rho)$ is $(M=G/H,d_c)$ as in theorem \ref{mman1}. If $d_c$ 
is Finslerian metric, defined by the norm $F_0$ on $D_H=T_HM$, 
then $\tau_xM$ is isometric to the normed vector space $(T_HM,F_0);$ 
otherwise $\tau_xM$ is isometric to a Carnot group $(C,d_{cc}),$ 
where normed vector spaces $(L_1,F_c)$ and $(D_H,F_0)$ are isometric.
\end{theorem} 

Let us note that it follows from theorem \ref{space} and the first 
statement of theorem \ref{tc} that $\tau_xM$ is a locally compact 
homogeneous space with inner metric which has $a$-similarities for 
every positive number $a$. Now other
statements of theorem \ref{tc} could be deduced from theorem \ref{ccg}.

\section{Homogeneous Finsler manifolds}

\begin{theorem}
\label{fins}\cite{BV-92}
A metric space $(M,\rho)$ is isometric to a homogeneous Finslerian manifold
if and only if $(M,\rho)$ is locally compact homogeneous space with inner
metric of finite topological dimension which is equal to its Hausdorff dimension.
\end{theorem}

\begin{proof}
The necessity of these conditions is well-known.

Sufficiency. By theorems \ref{man}, \ref{mman}, every locally compact 
homogeneous space with inner metric $(M,\rho)$ of finite topological 
dimension is a (sub-)Finslerian homogeneous manifold, defined by conditions 
from theorem \ref{mman1}. 

If $F_H$ from theorem \ref{mman1} is not Euclidean, then
we can find an Euclidean norm $F_{1H}$ on $D_H,$ invariant relative 
to the (linear) isotropy group of $G/H$ at $H\in G/H.$ Then there 
is a constant $c>1$ such that
$(1/c)F_{1H}\leq F_{H}\leq cF_{1H}.$ Now let $\rho_1=d_{1c}$ 
be $G$-invariant (sub-)Riemannian metric on $G/H$ defined by 
formula (\ref{dist}), where $F$ is $G$-invariant norm on $D$ 
such that $F(H)=F_{1H}.$ Then it is easy to see that
$(1/c)\rho_{1}\leq \rho\leq c\rho_{1}$ and therefore $(M,\rho)$ and $(M,\rho_1)$ 
have equal Hausdorff dimensions.

The space $(M,\rho)$ is Finslerian if and only if $(M,\rho_1)$ is 
Riemannian. To finish proof it is enough to apply for $(M,\rho_1)$ 
theorems \ref{tc}, \ref{ccg},
and known facts that so-called \textit{equiregular} connected smooth 
sub-Riemannian manifold $M$ and any its tangent cone $\tau_xM$ have equal 
Hausdorff dimensions, while the Hausdorff dimension of the Carnot group 
$(C,d_{cc})$ from theorem \ref{ccg} is equal to 
$$\sum_{k=1}^{m}k\dim(L_k)> \sum_{k=1}^{m}\dim(L_k)=\dim(T_H(G/H)),\quad m>1.$$
\end{proof}

\begin{theorem}
\cite{Ber-89.2}, \cite{Ber-88}
Every Lie group with bi-invariant (i.e. with left- and right-invariant) inner
metric is the Lie group with bi-invariant Finslerian metric.
\end{theorem}

\section{Homogeneous (sub-)Riemannian manifolds}

\begin{theorem}
\label{rim}
\cite{BerP}
A metric space $(M,\rho)$ is isometric to a homogeneous Riemannian manifold
if and only if $(M,\rho)$ is a locally compact homogeneous space with inner
metric of finite topological dimension which has the curvature $\geq K$
in A.~D.~Aleksandrov's sense for some $K\in \mathbb{R}$. 
\end{theorem}

Notice that there are different equivalent definitions of Aleksandrov 
spaces of curvature $\geq K$ \cite{Al},\cite{BGP}. The following 
definition belongs to the author.

\begin{definition}
\label{ale}
\cite{Ber-86}, \cite{BerP}
A space $M$ with an inner metric $\rho$ and with the local existence 
of shortest arcs is called the {\it Aleksandrov space} of
curvature $\geq K$, if locally any quadruple of points in $M$ is 
isometric to some quadruple of points in a simply connected complete 
Riemannian 3-manifold of some constant sectional curvature $k\geq K,$ 
where the number $k$ depends on considered quadruple of points.
\end{definition}

\begin{remark}
There are infinite dimensional compact homogeneous spaces with inner metric
of Aleksandrov curvature $\geq 0$ \cite{BerP}. A smooth Riemannian manifold
has Aleksandrov curvature $\geq K$ if and only if its sectional curvatures
$\geq K.$ The definition \ref{ale} is local, but every quadruple 
of points in geodesic Aleksandrov space of curvature $\geq K$ in a 
sense of this definition satisfies conditions from definition \ref{ale} 
\cite{BGP}. Some other conditions, in terms of orbits of 1-parameter 
subgroups of isometries, characterizing homogeneous Finsler and 
Riemannian manifolds, are given in papers \cite{Ber-89.2}, \cite{Ber-89.1}.   
\end{remark}

I don't know simple metric conditions, characterizing homogeneous sub-Riemannian
manifolds, aside as the Gromov-Hausdorff limits of homogeneous Riemannian
manifolds, when limits have different finite topological and Hausdorff dimensions. 

It is interesting that there is a probabilistic approach to solve 
this problem, at least in the case of left-invariant inner metrics on Lie groups. 

\begin{theorem}
\cite{Hey}, \cite{BGorb-14}
Left-invariant (sub-)Riemannian metrics on a connected Lie group are 
in 1-1 correspondence with symmetric Gaussian 1-parameter convolution semigroups of 
$\{e\}$-continuous, absolutely continuous with respect to left-invariant
Haar measure, probability measures on it.
\end{theorem}

Omitting details, we reference to exact definitions and theorem 
6.3.8 in book\cite{Hey} which characterizes generating 
infinitesimal (hypo-)elliptic operators of such semigroups. Notice 
that there is no mention to left-invariant (sub-)Riemannian metrics 
on Lie groups in \cite{Hey}.

\begin{problem}
It would be desirable to get a generalization of theorem 6.3.8 in \cite{Hey}
to the case of homogeneous manifolds $G/H$ and use it for 
(sub-)Riemannian geometry.
\end{problem}

\section{Homogeneous manifolds with integrable invariant distributions}

In this section we consider very natural problem: \textit{describe 
connected homogeneous manifolds such that every invariant inner metric 
on any of them is Finslerian}.

\begin{theorem}
\label{main}
\cite{Ber-89.1}, \cite{Ber-89.2}
Every $G$-invariant inner (geodesic) metric on the homogeneous space 
$G/H$ of a connected Lie group $G$ with a compact stabilizer 
$H\subset G$ is Finslerian if and only if

(A) Every $G$-invariant distribution on $G/H$ is integrable.

This is equivalent to the condition

(B) Every $Ad(H)$-invariant vector subspace $c$ in $\mathfrak{g},$ containing 
$\mathfrak{h},$ is a Lie algebra.

If $H$ is connected, in particular, if $G/H$ is simply connected, 
then the $Ad(H)$-invariance of the space $c$ is equivalent to the 
inclusion $[\mathfrak{h},c]\subset c.$
\end{theorem}

\begin{definition}
\label{miid}
The homogeneous manifold $G/H$ of a connected Lie group $G$ with 
a compact stabilizer $H$ is called \textit{homogeneous manifold 
with integrable invariant distributions}, shortly, HMIID, if it 
satisfies any of the equivalent conditions (A) or (B) from theorem~\ref{main}.
\end{definition}

\begin{theorem}
\label{lie}
\cite{Ber-89.1}, \cite{Ber-89.2}
The following conditions for a connected Lie group $G$ with the Lie algebra
$\mathfrak{g}$ are equivalent:

\begin{enumerate}
\item[1)] Every left-invariant inner metric on  the Lie group $G$ is Finslerian;

\item[2)] Every vector subspace of the Lie algebra $\mathfrak{g}$ is a Lie subalgebra in $\mathfrak{g};$

\item[3)] $\mathfrak{g}$ is one-dimensional or any two-dimensional vector
subspace in $\mathfrak{g}$ is a Lie subalgebra of $\mathfrak{g}$;

\item[4)] For any two elements $X,Y$ in $\mathfrak{g}$, the bracket $[X,Y]$ is a
linear combination of elements $X$ and $Y.$

\end{enumerate}
\end{theorem}

\begin{theorem}
\label{mil}
\cite{Mil}
If a Lie algebra $\mathfrak{g}$ satisfies condition 4) from theorem \ref{lie} then
\begin{enumerate}
\item[1)] there exists a linear map $l:\mathfrak{g}\rightarrow \mathbb{R}$ 
such that
\begin{equation}
\label{l}
[X,Y]=l(X)Y-l(Y)X, \quad X,Y \in \mathfrak{g};
\end{equation}

\item[2)] the kernel of the linear map $l$ is the maximal commutative ideal in 
$\mathfrak{g};$

\item[3)] $l=0$ if and only if $\mathfrak{g}$ is commutative Lie algebra;

\item[4)] if $l\neq 0$ then, up to an isomorphism, the Lie algebra 
$\mathfrak{g}$ has the form 
$$L_n = \mathbb{R} +_{\phi} \mathbb{R}^{n-1},\quad n\geq 2,$$ 
i.e. semidirect sum, prescribed by homomorphism 
$\phi : \mathbb{R} \to \text{End}(\mathbb{R}^{n-1})$, such that 
$\phi(1) = E_{n-1}$ is the unit matrix.
\end{enumerate}
\end{theorem}

\begin{theorem}
\label{mil1}
\cite{Mil}
Let $G_n$ be $n$-dimensional Lie group $G$ with the Lie algebra 
$\mathfrak{g}$, satisfying condition 4) from theorem \ref{lie}. Then
\begin{enumerate}
\item[1)] $G_n$ is commutative or

\item[2)] $G_n$ is isomorphic to the group of real $(n\times n)$ block matrices
\begin{equation}
\label{bd}
\left (\begin{array}{cc} aE_{n-1}&b \\0 & 1  \end{array}\right),
\end{equation}
where $E_{n-1}$ is the unit $(n-1)\times (n-1)$-matrix, $a$ is 
any positive number, $b$ is any $(n-1)$-vector-column, and
$0$ is the zero $(n-1)$-vector-row.
\end{enumerate}
\end{theorem}

\begin{theorem}
\label{mil2}
\cite{Mil}
Noncommutative Lie group $G$ has the Lie algebra $\mathfrak{g}$ 
satisfying condition 4) from theorem \ref{lie} if and only if 
any left-invariant Riemannian metric on $G$ has constant negative curvature.
\end{theorem}

\begin{theorem}
\label{sim}
\cite{Ber-89.1}, \cite{Ber-89.2}
Let $M$ be a connected Riemannian symmetric space, $G$ be the maximal connected
Lie group of isometries for $M$ with the stabilizer $H\subset G$ at a point 
$x\in M.$ Then any $G$-invariant inner metric on $G/H$ is Finslerian.
\end{theorem}

\begin{theorem}
\label{sii}
\cite{Ber-89.1}, \cite{Ber-89.2}
Let assume that $M = G/H$ (where $G$ is a connected Lie group and 
$H$ is its compact subgroup) be isotropy irreducible homogeneous 
spaces, i.e. $G/H$ has irreducible linear isotropy group. Then 
any $G$-invariant inner metric on $G/H$ is Finslerian.
\end{theorem}

\begin{theorem}
\label{comp}
\cite{Ber-92}, \cite{Ber-95}, \cite{Ber-96}
For any (compact) simply connected effective homogeneous space 
$G/H$ of a connected compact Lie group $G$ with closed
stabilizer $H$ the following conditions are equivalent:
\begin{enumerate}
\item[1)] all $G$-invariant distributions on $G/H$ are integrable;

\item[2)] the homogeneous space $G/H$ is isomorphic to a direct 
product of compact simply connected isotropy irreducible homogeneous spaces;

\item[3)] the space $G/H$ has normal type by Berard-Bergery \cite{Berg}, i.e. any
$G$-invariant Riemannian metric on $G/H$ is normal homogeneous in M.Berger's sense;

\item[4)] all $G$-invariant Riemannian metrics on $G/H$ have positive Ricci curvature;

\item[5)] all $G$-invariant Riemannian metrics on $G/H$ have positive 
scalar curvature.
\end{enumerate}
\end{theorem}

Simply connected irreducible (Riemannian) symmetric spaces  $G/H$ with connected Lie group $G$ and compact subgroup $H$ have been classified by E~Cartan in 1926, see \cite{Bes2}. They are (automatically) strictly isotropy irreducible, i.e. 
$H_0$ has irreducible isotropy representation; non-compact strictly isotropy irreducible homogeneous spaces are symmetric \cite{Man}. O.~Manturov \cite{Man} and J.~A.~Wolf \cite{W1} classified strictly isotropy irreducible homogeneous spaces; 
one needs to combine their results to get full classification; see also \cite{Bes2}. 

The author used no classification when he proved his results stated in this paper.  
It follows from previous statements that there is a full classification of
compact simply connected HMIID.

V.~V.~Gorbatsevich \cite{vgorvich1}, \cite{BGorb-14} studied  general 
homogeneous spaces with connected stabilizer subgroup from the class 
HMIID in detail. He described corresponding transitive Lie groups and stabilizer 
subgroups in the case when the transitive group is semisimple or solvable, 
and partly, in the case of general transitive Lie groups.

\end{document}